\documentclass[12pt]{amsart}
\title[Parameterizations of power-subanalytic sets]{Smooth parameterizations of power-subanalytic sets and compositions of Gevrey functions}
\author[S. Van Hille]{Siegfried Van Hille}
\address{KU Leuven, Celestijnenlaan 200B, 3001 Leu\-ven, Bel\-gium}
\email{siegfried.vanhille@kuleuven.be}
\date{\today}

\usepackage[margin=3cm]{geometry}
\usepackage{amsfonts}
\usepackage{amsmath}
\usepackage{amsthm}
\usepackage{amssymb}
\usepackage{graphicx}
\usepackage{amsopn}
\usepackage{parskip}
\usepackage{enumitem}
\usepackage{xcolor}
\usepackage{tikz}
\usetikzlibrary{matrix}
\usepackage[colorlinks]{hyperref}
\hypersetup{
	linkcolor = {blue},
	citecolor = {red}
}

\DeclareMathOperator{\R}{\mathbb{R}}

\DeclareMathOperator{\N}{\mathbb{N}}
\DeclareMathOperator{\Q}{\mathbb{Q}}
\DeclareMathOperator{\im}{\text{Im}}

\newcommand\green[1]{{\color{black}#1}}

\theoremstyle{definition}
\newtheorem{definition}{Definition}[section]

\theoremstyle{plain}
\newtheorem{theorem}{Theorem}[section]
\newtheorem{corollary}{Corollary}[theorem]
\newtheorem{lemma}[theorem]{Lemma}
\newtheorem{proposition}[theorem]{Proposition}
\newtheorem*{theorem*}{Theorem}
\newtheorem*{corollary*}{Corollary}

\begin{document}
\thanks{The author is partially supported by KU Leuven grant IF C14/17/083.}



\begin{abstract}
We show that if $X$ is an $m$-dimensional definable set in $\mathbb{R}_\text{an}^\text{pow}$, the structure of real subanalytic sets with real power maps added, then for any positive integer $r$ there exists a $C^r$-parameterization of $X$ consisting of $cr^{m^3}$ maps for some constant $c$. Moreover, these maps are real analytic and this bound is uniform for a definable family.
\end{abstract}

\maketitle

\section{Introduction}
A parameterization of a set $X \subset \R^n$ of dimension $m$ is a finite set of maps $$\{f_i: (0,1)^m \to X \mid i \in \{1,\ldots,N\}\}$$ for some $N \in \N$, such that the images of these maps cover $X$. Consider the following question: given any positive integer $r$, can one construct a parameterization of $X$ consisting of $r$ times continuously differentiable maps whose $C^r$-norm is bounded by $1$, i.e. a $C^r$-parameterization?

The interest in this question first arose in 1987 \green{in the study of entropy of dynamical systems in  \cite{Y87} and \cite{yomdin}. In these papers}, Yomdin sketched a $C^r$-parameterization result for semi-algebraic sets. A more refined explanation was given by Gromov in \cite{gromov}, a complete proof by Burguet in \cite{B08}. In their result the number of maps of the parameterization is shown to depend on combinatorial data defining $X$, i.e. $n$, $m$, $r$ and the complexity $\beta$, which is the maximum degree of the equations and inequalities necessary to define $X$. However, there is no explicit formula for the number of maps. Yomdin has written a nice survey on the study of parameterizations in \cite{survey}. 

Using o-minimality, a tool of model theory, Pila and Wilkie showed in 2006 that any bounded definable set in any o-minimal structure (containing the semi-algebraic sets) has a $C^r$-parameterization \cite{count}.  \green{However, in the special case that $X$ is semi-algebraic, they did not deduce a formula for the number of maps of the parametrization in terms of the complexity of $X$.}
The main result of \cite{count} bounds the number of rational points on a definable set $X$ up to a certain height, known as the counting theorem. To obtain this result, they constructed $C^r$-parameterizations for arbitrarily large $r$ and therefore it is interesting to know how large the parameterization becomes depending on $r$.

The dependence on $r$ has recently been investigated in \cite{ccs} by Binyamini and Novikov, who construct a $C^r$-parameterization consisting of $cr^m$ maps for subanalytic sets. Here, $c$ is a constant that depends on $X$.  Moreover, they show that $c$ depends polynomially on the complexity of $X$, if $X$ is semi-algebraic. Around the same time, Cluckers, Pila and Wilkie proved in \cite{unif} a $C^r$-parameterization result for power-subanalytic sets using $cr^d$ maps, where $c$ and $d$ are constants depending on $X$, and this result is uniform.  

In this paper, we built on the methods of \cite{unif} to further refine their result in Theorem (\ref{maintheorem}). It is weaker than the result of \cite{ccs} for subanalytic sets, but holds \green{for} the larger class of power-subanalytic sets.

\begin{theorem*}
If $X \subset [-1,1]^n$ is power-subanalytic of dimension $m \leq n$, then for any positive integer $r$ there exists a $C^r$-parameterization of $X$ consisting of $cr^{m^3}$ maps whose $C^r$-norm is bounded by $1$. Moreover, if $X$ belongs a power-subanalytic family of such sets, the constant $c$ holds for all members of the family.
\end{theorem*}

The theorem above and the results of \cite{unif} actually hold for structures $\R^K_\mathcal{F}$, which were studied by Miller in \cite{mil}. These structures expand the semi-algebraic sets with restricted analytic functions in a Weierstrass system $\mathcal{F}$ and power maps $x \mapsto x^\mu$ for any $\mu \in K$, where $K$ is a subfield of the field of exponents of $\mathcal{F}$ (see the introduction of Section \ref{main section} for precise definitions). The structures $\R$, $\R_\text{an}$ and $\R^\text{pow}_\text{an}$, corresponding to the semi-algebraic, subanalytic and power-subanalytic sets respectively, are all examples of this class of structures. The main result of \cite{mil}, is a preparation theorem for functions definable in these structures, which is one of the main ingredients for the proof of the main result. 

In 2011, Jones, Miller and Thomas have shown in \cite{mildpara} that any set definable in (any reduct of) $\mathbb{R}_\text{an}$ expanding the real field has an $(A,0)$-mild parameterization for some $A>0$, i.e. a parameterization consisting of maps that are $(A,0)$-mild. \green{The precise definition of a mild function is given in Section \ref{mild functions}.} Informally, for $C \geq 0$, an $(A,C)$-mild function is a smooth function with good bounds on the derivatives depending on $A$ and $C$.  For any integer $r > 0$, these bounds allow us to bound the $C^r$-norm by $1$ after a suitable substitution. Since we only control the derivatives up to order $r$, their result is stronger. However, by an elementary example of Yomdin \cite[Proposition 3.3]{example}, their result cannot be made uniform. Indeed, the family of hyperbolas $\{ xy = t^2 \mid (x,y) \in (0,1)^2, t \in (0,1)\}$ does not have a uniform $(A,0)$-mild parameterization, i.e. the number of charts will depend on the parameter $t$. Recently it has been shown in \cite{expmild} that it has a uniform $(A,C)$-mild parameterization for all $C>0$.

The paper is organized as follows. In Section \ref{mild functions}, we prove various properties of mild functions using only standard techniques of real analysis. A key result on the composition of these functions, Theorem (\ref{thmgevrey}), permits us to make all of the results of \cite{unif} on mild functions explicit and thus also the constant $d$. This key result is in fact an old result on Gevrey functions \cite{ge}. We will give the original proof of this result, but in full generality, and reformulate it in terms of mild functions. In Section \ref{main section} we provide the necessary background in model theory to state the main theorem and we prove the main theorem. Finally we explain that when $m \geq 2$, we can in fact obtain a parameterization consisting of $cr^{m^2(m-1)}$ maps.

\section{Mild functions} \label{mild functions}
Mild functions are a class of $C^\infty$-maps introduced by Pila in \cite{milddef}. \green{The upper bounds on the derivatives of a mild function are very suitable to use the determinant method. This method, which Pila developed with Bombieri \cite{determinant}, a useful tool in diophantine geometry. For example, the determinant method is used in the proof of the counting theorem in \cite{count}.} 

Cluckers, Pila and Wilkie introduced several variants to mild functions in \cite{unif}. Since in their result (and our main theorem), they only cared about derivatives up to order $r$\green{, they defined functions that are mild up to order $r$. These functions are $C^{r}$-maps that satisfy the same bounds on the derivatives as mild functions, but up to order $r$. In fact, the maps are actually smooth in their work.} 


Since we will encounter many compositions of mild maps, we need a result on compositions of these maps. \green{This result on compositions will allow us to find the number of charts in Section \ref{main section}.} One can think of mild functions as functions of some class indexed by a real number $C \geq 0$. When $C = 0$, these maps are real analytic, when $C= +\infty$, they are real smooth. Using the theory of Gevrey functions \cite{ge}, we show that compositions of mild functions of class $C$ are again of class $C$.

\vspace{0.3cm}We start with some definitions of multidimensional calculus. Throughout this section we will work with maps $f = (f_1,\ldots,f_n): U \subset \R^d \to \R^n$, where $U$ is always assumed to be open in $\R^d$. We say that $f: U \subset \R^d \to \R$ is $C^r$ for a natural number $r$ (or $+\infty$) if $f$ is $r$ times continuously differentiable on $U$ and $f = (f_1,\ldots,f_n): U \subset \R^d \to \R^n$ is $C^r$ if $f_1,\ldots,f_n$ are all $C^r$. For any $i \in \{1,\ldots,d\}$, denote $(\partial f / \partial x_i) = (\partial f_1 / \partial x_i,\ldots, \partial f_n / \partial x_i)$. For any multi-index $\nu = (\nu_1,\ldots,\nu_d) \in \N^d$, map $f: U \to \R^n$ and $x = (x_1,\ldots,x_d) \in U$ we set:
\begin{align*}
|\nu| &= \nu_1 + \ldots + \nu_d \\
\nu! &= \nu_1!\cdots \nu_d! \\
x^\nu &= \prod_{i=1}^d x_i^{\nu_i} \\
f^{(\nu)} &= \frac{\partial^{|\nu|}}{\partial x_1^{\nu_1}\cdots \partial x_d^{\nu_d}}f
\end{align*}
where by definition $0! = 1$, and $0^0 = 1$. 

\begin{definition}[$C^r$-norm]\label{norm}
Suppose that $f: U \subset \R^d \to \R$ is $C^r$. Then we define the $C^r$-norm $|\cdot|_r$ of $f$ as follows:
\begin{equation*}
|f|_r = \sup_{x \in U} \sup_{\substack{|\nu| \leq r \\ \nu \in \N^d}} \frac{|f^{(\nu)}(x)|}{|\nu|!}.
\end{equation*}
Note that $|f|_r$ can be $+\infty$. We define the $C^r$-norm of a map $f:U \subset \R^d \to \R^n$ to be the maximum of the $C^r$-norms of the component functions $f_1,\ldots,f_n$.
\end{definition}

This is the norm used in \cite{ccs}. In \cite{unif} they did not divide by $|\nu|!$. This yields equivalent norms but has an impact on the exponent of $r$ in the main theorem (see Lemma (\ref{ACsub})).

\begin{definition}[Mild functions] \label{defmild}
Suppose that $A,B>0$, $C \geq 0$ are real numbers and $f: U \subset \R^d \to \R$. Then $f$ is called $(A,B,C)$-mild if it is $C^\infty$ and if for all $\nu \in \N^d$ and $x \in U$: 
\begin{equation*}
\left| f^{(\nu)}(x)\right| \leq B^{C+1}A^{|\nu|}|\nu|!^{C+1}.
\end{equation*}
A map $f: U \subset \R^d \to \R^n$ is $(A,B,C)$-mild if all component functions $f_1,\ldots,f_n$ are $(A,B,C)$-mild. If $B \leq 1$, we simply say that $f$ is $(A,C)$-mild and if we say that $f$ is mild, then we mean that $f$ is $(A,C)$-mild for some $A>0$ and $C\geq 0$.
\end{definition}

This definition is slightly different than the original one by Pila in \cite{milddef} (\cite{milddef2} for the multivariate case), where the bounds on the derivatives for an $(A,C)$-mild function are given by $\nu!(A|\nu|^C)^{|\nu|}$. Now, due to the following inequalities:
\begin{align*}
\nu! &\leq |\nu|! \leq d^{|\nu|}\nu! \\
|\nu|! &\leq |\nu|^{|\nu|} \leq e^{|\nu|}|\nu|!
\end{align*}
(where the number $e$ appears as a result of Stirling's formula), one sees that the definitions of $(A,C)$-mild coincide in the following way: $f$ is $(A_1,C)$-mild for some $A_1>0$ as in our definition if and only if $f$ is $(A_2,C)$-mild for some $A_2 > 0$ as in \cite{milddef2}. The reason to adjust the definition is to keep the proof of Theorem (\ref{thmgevrey}) as simple as possible.

Following \cite{unif}, we will use `up to order $r$' versions of many definitions and theorems, for example the following definition.

\begin{definition}[Mild up to order $r$] \label{defmildor}
Suppose that $A,B>0$, $C \geq 0$ are real numbers, $r > 0$ an integer or $+\infty$ and $f: U \subset \R^d \to \R$. Then $f$ is called $(A,B,C)$-mild up to order $r$ if it is $C^r$ and if for all $\nu \in \N^d$ with $|\nu| \leq r$ and $x \in U$: 
\begin{equation*}
\left| f^{(\nu)}(x)\right| \leq B^{C+1}A^{|\nu|}|\nu|!^{C+1}.
\end{equation*}
A map $f: U \subset \R^d \to \R^n$ is $(A,B,C)$-mild up to order $r$ if all component functions $f_1,\ldots,f_n$ are $(A,B,C)$-mild up to order $r$. If $B \leq 1$, we simply say that $f$ is $(A,C)$-mild up to order $r$ and if we say that $f$ is mild up to order $r$, then we mean that $f$ is $(A,C)$-mild up to order $r$ for some $A>0$ and $C \geq 0$.
\end{definition}

Note $(A,B,C)$-mild up to order $+\infty$ agrees with $(A,B,C)$-mild. The following corollary is an immediate consequence of the definitions.

\begin{corollary}
Suppose $f: U \subset \R^d \to \R^n$ is $C^r$. Then $f$ has $C^r$-norm less than or equal to $B$ if and only if $f$ is (1,B,0)-mild up to order $r$.
\end{corollary}

For $(A,C)$-mild functions up to order $r$, one can bound the $C^r$-norm by $1$ with an easy substitution. It will be used to count the numbers of maps in the proof of the main theorem. More precisely, we will construct a parameterization consisting of $(A,0)$-mild maps up to order $r$.

\begin{lemma}\label{ACsub}
Suppose that $f: U \subset (0,1)^d \to [-1,1]$ is $(A,C)$-mild up to order $r$. Let $P = (P_1,\ldots,P_d)$ be any point in $(0,1)^d$ and consider the map $$\psi: (x_1,\ldots,x_d) \mapsto \left(\frac{x_1}{Ar^C}+P_1,\ldots,\frac{x_d}{Ar^C}+P_d\right).$$ If $V = \psi^{-1}(U)$ then $(f \circ \psi): V \to [-1,1]$ has $C^r$-norm bounded by $1$.
\end{lemma}
\begin{proof}
This is a direct calculation using the chain rule.
\end{proof}

This lemma is nearly identical to \cite[Lemma 4.1.3]{unif}. The difference is the exponent of $r$, which is $C+1$ in \cite{unif}, because they use the following norm:
$$\sup_{x \in U} \sup_{\substack{|\nu| \leq r \\ \nu \in \N^d}} |f^{(\nu)}(x)|.$$ Since this norm is stronger, the number of maps in the $C^r$-parameterization is larger.

We conclude this section with the following class of examples which will show up in Section \ref{main section}. For a proof we refer to \cite[Proposition 2.2.10]{primer}.

\begin{lemma}\label{unitmild}
Suppose that $f: U \subset \R^d \to \R$ is analytic on an open neighborhood of the topological closure $\overline{U}$ of $U$. Then $f$ is (A,B,0)-mild for some $A,B>0$.
\end{lemma}

\subsection{Compositions of mild functions}
In this section we show that if $f$ and $g$ are $(A_f,B_f,C)$- and $(A_g,B_g,C)$-mild respectively, then the composition $f\circ g$ is $(A,B,C)$-mild for some explicit $A$ and $B$. Of course, we will need a multivariate version of a formula for arbitrary derivatives of a composite function. This is known as the Fa\`a di Bruno formula, which has first been proved in \cite{faa} and we reformulate it here. A proof can also be found in \cite[Theorem 1.3.2]{primer}.

\begin{theorem}\label{thmfaa}
Suppose that $n$ is a positive integer, $V \subset \R^d$ and $U \subset \R^e$ are open, $f: V \to \R$, $g: U \to V$ and that $f$ and $g$ are $C^n$. For any $x \in U$ and $\nu \in \N^e$ with $|\nu| = n$ we have that $$(f \circ g)^{(\nu)}(x) = \sum_{1 \leq |\lambda| \leq n} f^{(\lambda)}(g(x)) \sum_{s = 1}^n \sum_{p_s(\nu,\lambda)} \nu! \prod_{j = 1}^s \frac{(g^{(l_j)}(x))^{k_j}}{k_j!(l_j!)^{{|k_j|}}}$$ where $p_s(\nu,\lambda)$ is the set consisting of all $k_1,\ldots,k_s \in \N^d$ with $|k_i| > 0$ and $l_1,\ldots,l_s \in \N^e$ with $0 \prec l_1 \prec \ldots \prec l_s$ such that: $$ \sum_{i = 1}^sk_i = \lambda$$ and $$\sum_{i = 1}^s |k_i|l_i = \nu.$$ Here $l_i \prec l_{i+1}$ means that $|l_i| < |l_{i+1}|$ or, if $|l_i| = |l_{i+1}|$, then $l_i$ comes lexicographically before $l_{i+1}$.
\end{theorem}

This formula is the heart of all of the proofs of results on compositions of mild functions that will follow. In \cite{unif}, Cluckers, Pila and Wilkie use this formula to deduce that the composition of mild functions is mild. However, they show this by roughly estimating the sum and hence they obtain no explicit formulas for $A,B$ and $C$. In his paper \cite{ge}, Gevrey introduces `functions of class $\alpha$' (see below), later known as Gevrey functions, and showed that the functions of class $\alpha$ with $\alpha \geq 1$ are closed under composition. In fact, Gevrey proved his result only where one of the functions involved has only one variable. The proof uses the Fa\`a di Bruno formula in an essential way. Even though we refer to a more recent paper for the proof of this formula, the formula was already known for a long time. Because the result on Gevrey functions will immediately yield our result on mild functions and the technique is important for all other proofs, we will give a full proof of Gevrey's theorem for functions in arbitrary many variables, that is, the general version of what he proved.

\begin{definition}[Gevrey functions]\label{defgevrey}
Suppose that $\alpha \geq 0$ and $f: U \subset \R^d \to \R$. We say that $f$ is a Gevrey function of class $\alpha$ if it is $C^\infty$ and for all $\nu \in \N^d$ and $x \in U$ we have: $$\left| f^{(\nu)}(x) \right| \leq \left(M \frac{|\nu|!}{R^{|\nu|}}\right)^\alpha$$ for some $M,R > 0$. A map $f: U \subset \R^d \to \R^n$ is a Gevrey function of class $\alpha$ if all of its component functions are Gevrey functions of class $\alpha$.
\end{definition}

The following proposition, together with Theorem (\ref{thmgevrey}), motivates why we adjusted the definition of a mild function, as we explained below Definition (\ref{defmild}).

\begin{proposition}\label{mildgevrey}
Suppose that $f: U \subset \R^d \to \R$ is $C^\infty$, then $f$ is $(A,B,C)$-mild for some $A,B >0$ if and only if it is a Gevrey function of class $C+1$.
\end{proposition}
\begin{proof}
If one sets $R = A^{-1/(C+1)}$ and $M = B$, then it follows immediately that the bounds on the derivatives are the same.
\end{proof}

Of course, one could define a Gevrey function of class $\alpha$ up to order $r$ and the result above has an up to order $r$ version \green{and thus, so do the theorem and the subsequent corollary below.}

\begin{theorem}[Gevrey \cite{ge}]\label{thmgevrey}
Suppose that $f:V\subset \R^d \to \R$ and $g: U \subset \R^e \to V$ are Gevrey functions of class $\alpha \geq 1$, then $f \circ g$ is a Gevrey function of class $\alpha$.
\end{theorem}
\begin{proof}
The idea of this ingenious proof of Gevrey is as follows. For some specific $F$ and $G$, we can explicitly compute both sides of the Fa\`a di Bruno formula at some well chosen point since all derivatives of $F \circ G$ will be easy to compute. In turns out that the values of these derivatives at this point are exactly the bounds we have for the derivatives of our given maps $f$ and $g$. Now let $x \in U$ be arbitrary. Using the Fa\`a di Bruno formula and the triangle inequality to bound $|(f \circ g)^{(\nu)}(x)|$, we obtain the same sum of positive terms where the bounds on the derivatives of $f$ and $g$ occur. Moreover, this sum does not depend on $x$, since the bounds of a Gevrey function are uniform.

To define the maps $F$ and $G$, we need the constants that bound the derivatives of $f$ and $g$. Since they're Gevrey functions of class $\alpha$, we have:
\begin{equation*}
\left| f^{(\nu)}(x) \right| \leq \left(M_f \frac{|\nu|!}{R_f^{|\nu|}}\right)^\alpha
\end{equation*}
and
\begin{equation*}
\left| g_{i}^{(\nu)}(x) \right| \leq \left(M_g \frac{|\nu|!}{R_g^{|\nu|}}\right)^\alpha
\end{equation*}
for $i = 1,\ldots,d$. The maps that do the trick are the following $F$ and $G = (G_1,\ldots,G_d)$:
\begin{equation*}
F = \frac{M_fR_f}{(R_f+dM_g) - (x_1+\ldots+x_d)}
\end{equation*}
and for $i = 1,\ldots,d$:
\begin{equation*}
G_i = \frac{M_gR_g}{R_g - (x_1+\ldots+x_e)}.
\end{equation*}
We first compute an arbitrary derivative of $G_i$ for some $i \in \{1,\ldots,d\}$. Set $S(x) = x_1+\ldots + x_e$. Then we have for all $j \in \{1,\ldots,e\}$ that $(\partial/\partial x_j)(S(x)) = 1$. Writing $G_i(x) = M_gR_g(R_g-S(x))^{-1}$ we easily obtain that for any $\lambda \in \N^e$: $$G_i^{(\lambda)}(x) = (M_gR_g)|\lambda|!(R_g-S(x))^{-(1+|\lambda|)}.$$ Hence we get:
\begin{equation}\label{gvalue}
G_i^{(\lambda)}(0) = M_g \frac{|\lambda|!}{R_g^{|\lambda|}}.
\end{equation}
Noticing that $G(0) = (M_g,\ldots,M_g)$, we see that for all $\lambda \in \N^d$ we obtain similarly:
\begin{equation}\label{fvalue}
F^{(\lambda)}(G(0)) = M_f\frac{|\lambda|!}{R_f^{|\lambda|}}.
\end{equation}
Next, one checks that: $$(F \circ G)(x) = \frac{M_fR_f(R_g-S(x))}{R_fR_g - (R_f+dM_g)S(x)}$$ and that for any $\lambda \in \N^e$ with $|\lambda| \geq 1$: 
\begin{equation*}
(F \circ G)^{(\lambda)}(x) = \frac{dM_fM_gR_fR_g}{R_f+dM_g}\frac{(R_f+dM_g)^{|\lambda|}|\lambda|!}{(R_fR_g - (R_f+dM_g)S(x))^{|\lambda|+1}}.
\end{equation*}
Finally we get that: 
\begin{equation}\label{fgvalue}
(F \circ G)^{(\lambda)}(0) = M\frac{|\lambda|!}{R^{|\lambda|}}
\end{equation}
with $M = \frac{dM_fM_g}{R_f+dM_g}$ and $R = \frac{R_fR_g}{R_f+dM_g}$. Hence if we plug in values as in (\ref{fvalue}) and (\ref{gvalue}) in the sum of Theorem (\ref{thmfaa}), then we know that it is equal to (\ref{fgvalue}). More precisely, we get the following equalities:
\green{\[
M\frac{|\nu|!}{R^{|\nu|}} = (F \circ G)^{(\nu)}(0) = \sum_{1 \leq |\lambda| \leq n} M_f\frac{|\lambda|!}{R_f^{|\lambda|}} \sum_{s = 1}^n \sum_{p_s(\nu,\lambda)} \nu! \prod_{j = 1}^s \frac{\left(M_g \frac{|l_j|!}{R_g^{|l_j|}}\right)^{|k_j|}}{k_j!(l_j!)^{|k_j|}}.
\]}
Now first suppose $\alpha = 1$. To compute an upper bound on $|(f \circ g)^{(\nu)}(x)|$, use the Fa\`a di Bruno formula, the triangle inequality and finally the bounds on the derivatives of $f$ and $g$ to obtain the sum on the right hand side of these equations. We now know it equals the desired form on the left hand side. This finishes the proof if $\alpha = 1$. If $\alpha > 1$, you additionally use that $r^\alpha + s^\alpha \leq (r+s)^\alpha$ $(r,s \geq 0)$, which yields the result.
\end{proof}

Combining this with Proposition (\ref{mildgevrey}) one deduces from this proof the following result on mild functions.

\begin{corollary}\label{compformula}
Suppose that $f: V \subset \R^d \to \R$ is $(A_f,B_f,C)$-mild and that $g: U \subset \R^e \to V$ is $(A_g,B_g,C)$-mild. Then $f\circ g$ is $(A,B,C)$-mild where:
\begin{align*}
A &= A_fA_g(A_f^{-1/(C+1)}+dB_g)^{C+1}, \\
B &= \frac{dB_fB_g}{A_f^{-1/(C+1)}+dB_g} < B_f.
\end{align*}
In particular, if $f$ is $(A_f,0)$-mild and $g$ is $(A_g,0)$-mild then $f \circ g$ is $(A,0)$-mild with:
\begin{equation*}
A = A_g(dA_f +1).
\end{equation*}
\end{corollary}

\subsection{Weakly mild functions and power substitutions}
Many nice functions are not mild. For instance, if $f$ is analytic on $U$, this still isn't sufficient: consider for example the function $x \mapsto x^{1/2}$ on $(0,1)$. So one can not weaken the conditions of Lemma (\ref{unitmild}). Maps of this form will be crucial later on. \green{They are also examples} of the next definition, a variant on mild functions that has been introduced by Cluckers, Pila and Wilkie in \cite{unif}. By a reparameterization, which we will call a power substitution, we can make these functions mild up to any order if they satisfy an additional condition. 

\begin{definition}[Weakly mild functions]
Suppose that $A,B>0$, $C \geq 0$ are real numbers, $r>0$ an integer or $+\infty$ and $f: U \subset (0,1)^d \to \R$. Then $f$ is called weakly $(A,B,C)$-mild up to order $r$ if it is $C^r$ and if for all $\nu \in \N^d$ with $|\nu| \leq r$ and $x \in U$: 
\begin{equation*}
\left| f^{(\nu)}(x)\right| \leq \frac{B^{C+1}A^{|\nu|}|\nu|!^{C+1}}{x^{\nu}}.
\end{equation*}
A map $f = (f_1,\ldots,f_n): U \subset (0,1)^d \to \R^n$ is weakly $(A,B,C)$-mild up to order $r$ if all component functions $f_1,\ldots,f_n$ are weakly $(A,B,C)$-mild up to order $r$. If $B \leq 1$, we simply say that $f$ is weakly $(A,C)$-mild up to order $r$ and if we say that $f$ is weakly mild up to order $r$, then we mean that $f$ is weakly $(A,C)$-mild up to order $r$ for some $A>0$ and $C \geq 0$. Finally, if $r = +\infty$, we might just say weakly $(A,B,C)$-mild or weakly mild.
\end{definition}

Using the technique of the proof of Theorem (\ref{thmgevrey}), we can make the result of \cite[Proposition 4.1.5]{unif} more explicit in the proposition below. In particular we have an explicit formula for $\tilde{A}$ and the constant $C$ is preserved.

\begin{proposition}[Power substitution]\label{powersub}
Let \green{$A \geq 1$}, $B>0$ and $C \geq 0$ be real numbers and suppose that $f: U \subset(0,1)^d \to \R$ is a map such that $f$ and all first order derivatives of $f$ are weakly $(A,B,C)$-mild. Let $r>0$ be an integer and define $\phi: (0,1)^d \to (0,1)^d$ by: $$(x_1,\ldots,x_d) \mapsto (x_1^{n_1},\ldots,x_d^{n_d})$$ with $n_i \geq r$ for all $i \in \{1,\ldots,d\}$, and denote $V = \phi^{-1}(U)$. Then the map $f \circ \phi: V \to \R$ is $(\tilde{A},B,C)$-mild up to order $r$ with $\tilde{A} = NA(d+1)^{C+1}$, with $N = \max(n_1,\ldots,n_d)$.
\end{proposition}

\green{\emph{Remark: if $f$ satisfies the conditions above, but with $0<A<1$ instead, one can apply the proposition after enlarging $A$ to $1$.}}

\begin{proof}
We first take a look at the derivatives of $\phi$. Take any component function $\phi_i$ of $\phi$. Because it only depends on $x_i$, we just have to bound the following derivatives:
\[
 \left(\frac{\partial}{\partial x_i}\right)^k\phi_i (x) = n_i\cdots(n_i-(k+1))x_i^{n_i-k} \leq N^k x_i^{n_i-k}\leq N^k.
 \]
Thus, $\phi$ is $(N,0)$-mild. However, we will use the sharper bound (involving the power $n_i-k$ of $x_i$) to deal with the derivatives of $f$. By Theorem (\ref{thmfaa}) and the triangle inequality, we have for every $x \in V$ that
\[
|(f \circ \phi)^{(\nu)}(x)| \leq \sum_{1 \leq |\lambda| \leq n} |f^{(\lambda)}(\phi(x))| \sum_{s = 1}^n \sum_{p_s(\nu,\lambda)} \nu! \prod_{j = 1}^s \frac{|(\phi^{(l_j)}(x))^{k_j}|}{k_j!(l_j!)^{|k_j|}}.
\]

Consider one term in this sum, so we also have a fixed $\lambda \in \N^d$ with $1 \leq |\lambda| \leq |\nu|$, \green{$s \in \{1,\ldots,n\}$} and  $(k_1,\ldots,k_s ,l_1,\ldots,l_s) \in p_s(\nu,\lambda)$. Then, for any $i \in \{1,\ldots,d\}$ and $j \in \{1,\ldots,s\}$, we have:
\begin{equation*}
\left|  \left(\phi_i^{(l_j)}(x)\right)^{k_{j,i}} \right| \leq N^{l_{j,i}k_{j,i}}x_i^{(n_i-l_{j,i})k_{j,i}}
\end{equation*}
and thus the power of $x_i$ coming from the product over $j$ is $\sum_{j = 1}^s (n_i-l_{j,i})k_{j,i}$. 

We now compute the `negative' contribution to the powers of $x_i$ in this term coming from $f^{(\lambda)}(\phi(x))$. Write $\lambda = \lambda' + \beta$ for some $\beta \in \N^d$ with $|\beta| = 1$ \green{and let $k \in \{1,\ldots,d\}$ be the unique index such that $\beta_{k}=1$}. For any choice of $\beta$ we have the following (thus we may pick some particular $\beta$, \green{equivalently, make a choice for $k$}, later):
\[
\left|f^{(\lambda)}(\phi(x)) \right| = \left| (f^{(\beta)})^{(\lambda')}(\phi(x)) \right| \leq B^{C+1}A^{|\lambda'|}|\lambda'|!^{C+1}\frac{1}{\phi(x)^{\lambda'}} \leq B^{C+1}A^{|\lambda|}|\lambda|!^{C+1}\frac{1}{\phi(x)^{\lambda'}}.
\]
Hence the power of $x_i$ coming from $f^{(\lambda)}(\phi(x))$ in the term is equal to $-n_i\lambda'_i$. We can now compute the total power of $x_i$ occurring in the term. The total power of $x_i$ when $i \neq k$ is (then $\lambda_i' = \lambda_i$):
\[
\sum_{j = 1}^s \left((n_i-l_{j,i})k_{j,i}\right) - n_i\lambda_i = \sum_{j = 1}^s \left((n_i-l_{j,i})k_{j,i}\right) - n_i\sum_{j = 1}^s k_{j,i} = -\sum_{j = 1}^s k_{j,i}l_{j,i}.
\]
When $i = k$ (then $\lambda'_k = \lambda_k-1$) one computes in the same way that the power of $x_k$ is $n_k - \sum_{j = 1}^s k_{j,k}l_{j,k}$. We now pick $k$ such that $$x_k = \min_{i: \lambda_i \neq 0} x_i.$$ Then we have for all $i \in \{1,\ldots,d\}$: $$x_i^{-\sum_{j = 1}^sk_{j,i}l_{j,i}} \leq \green{x_k^{-\sum_{j = 1}^s k_{j,i}l_{j,i}}} \leq x_k^{-\sum_{j = 1}^s |k_j|l_{j,i}}$$ and we can bound the product of all $x_i$ and their powers by:
\begin{equation*}
x_k^{n_k - \sum_{i = 1}^d \sum_{j = 1}^s |k_j|l_{j,i}} = x_k^{n_k - \sum_{i = 1}^d \nu_i} = x_k^{n_k - |\nu|} \leq 1
\end{equation*}
since $n_k \geq r$ and $|\nu| \leq r$. \green{So we can bound this term independent of $x$.}

\green{Putting everything together we obtain that:
\begin{align*}
|(f \circ \phi)^{(\nu)}(x)| &\leq \sum_{1 \leq |\lambda| \leq n} B^{C+1}A^{|\lambda|} |\lambda|!^{C+1} \sum_{s = 1}^n \sum_{p_s(\nu,\lambda)} \nu! \prod_{j = 1}^s \frac{N^{|l_{j}||k_{j}|}}{k_j!(l_j!)^{|k_j|}} \\
&\leq \sum_{1 \leq |\lambda| \leq n} B^{C+1}A^{|\lambda|} |\lambda|!^{C+1} \sum_{s = 1}^n \sum_{p_s(\nu,\lambda)} \nu! \prod_{j = 1}^s \frac{(N^{|l_{j}|} |l_{j}|!^{C+1})^{|k_{j}|}}{k_j!(l_j!)^{|k_j|}}.
\end{align*}
This is the same upper bound as the one that one would find for the composition of an $(A,B,C)$-mild function up to order $r$ and an $(N,1,C)$-mild function up to order $r$ using Proposition (\ref{mildgevrey}) and then the proof of Theorem (\ref{thmgevrey}). Therefore one can use the formula of} Corollary (\ref{compformula}) to find the expression for $\tilde{A}$.
\end{proof}

The extra condition that all first order derivatives are weakly mild is crucial and cannot be omitted.  For instance, consider again the map $x \mapsto x^{1/2}$. Then we see that by composing with the power map $x \mapsto x^3$, we do not obtain a map that is mild up to order $3$. Of course, using the power map $x \mapsto x^2$, it becomes mild (up to order $+\infty$). This observation will be the key to slightly improve the main theorem in Section \ref{main section}.
 
Even though we can say something about the derivatives of a composition of weakly mild functions, it is not necessarily weakly mild. We do have the following result.

\begin{proposition}\label{thmweakmild}
Suppose that $f: V \subset \R^d \to \R$ is $(A_f,B_f,C)$-mild up to order $r$ and that $g: U \subset (0,1)^e \to V$ is weakly $(A_g,B_g,C)$-mild up to order $r$. Then $f \circ g$ is weakly $(A,B,C)$-mild up to order $r$, where $A$ and $B$ are as in Corollary (\ref{compformula}).
\end{proposition}
\begin{proof}
The strategy of the proof is completely the same as in Proposition (\ref{powersub}): one just checks that in each term one obtains exactly a factor $\frac{1}{x^{\nu}}$ when bounding $(f \circ g)^{(\nu)}(x)$ using Theorem (\ref{thmfaa}).
\end{proof}

We conclude this section with a result on the product of (weakly) mild functions. In this proposition, one can replace $(0,1)^d$ by $\R^d$ in the case of mild functions.

\begin{proposition}\label{product}
Suppose that $f_1,\ldots,f_l: U \subset (0,1)^d \to \R$ are (weakly) $(A,B,C)$-mild up to order $r$ for some $A,B > 0$ and $C \geq 0$. Then the product $f_1 \cdots f_l$ is (weakly) $(lA,B^l,C)$-mild up to order $r$.
\end{proposition}
\begin{proof}
As in \cite[Proposition 2.6]{milddef2}, we have that:
\[
|(f_1 \cdots f_l)^{(\nu)}(x)| \leq \sum_{\nu_1+\ldots +\nu_l = \nu} \text{Ch}(\nu_1,\ldots,\nu_l)\prod_{i = 1}^l |f_i^{(\nu_i)}(x)|,
\]
where the sum runs over all $\nu_1,\ldots,\nu_l \in \N^d$ such that $\nu_1+\ldots+\nu_l = \nu$ and $\text{Ch}(\nu_1,\ldots,\nu_l)$ is a constant depending on $\nu_1,\ldots,\nu_l$. If $f_1,\ldots,f_l$ are $(A,B,C)$-mild, we find:
\[
\prod_{i = 1}^l |f_i^{(\nu_i)}(x)| \leq \prod_{i = 1}^l B^{C+1}A^{|\nu_i|}|\nu_i|!^{C+1} \leq B^{l(C+1)} A^{|\nu|} |\nu|!^{C+1}.
\]
If $f_1,\ldots,f_l$ are weakly $(A,B,C)$-mild, one obtains an extra factor $1/x^\nu$ as desired. Finally, in Section 3.2 of \cite{unif}, it is shown that $\sum_{\nu_1+\ldots+\nu_l = \nu} \text{Ch}(\nu_1,\ldots,\nu_l) \leq l^{|\nu|}$.
\end{proof}

\section{The $C^r$-parameterization theorem} \label{main section}
In this section we give a precise definition of the structure wherein the family $X_T$, $T \subset \R^k$, of our main theorem should be definable. This structure is o-minimal, a powerful tool of model theory, and thus enables us to use the cell decomposition theorem, see \cite{dries}. Furthermore, in \cite{mil} Miller has shown a preparation theorem for definable functions in this structure. Combining these results, we obtain a very strong parameterization theorem from which we will easily deduce the main theorem with the results of the last section.

\vspace{0.3cm} We start with the necessary definitions of model theory and the result of Miller.  In his paper definability (in $\R^n$) is with respect to the following language. Let $\mathcal{L}_\text{r} = \{+,-,\cdot,<,0,1\}$ be the language of ordered rings and expand it with a symbol for each of the following functions: 
\[
\tilde{f}(x) = \begin{cases} f(x) & \text{ if $x \in [-1,1]^n$} \\ 0 & \text{ elsewhere,}\end{cases}
\]
where $f: U \to \R$ is a real analytic function on an open neighborhood $U$ of $[-1,1]^n$. Denote this language $\mathcal{L}_\text{an}$. If $X$ is $\mathcal{L}_\text{an}$-definable in $\R^n$, then it is called (globally) subanalytic. Remember that a function is definable if its graph is a definable set. Next, if we expand $\mathcal{L}_\text{an}$ with a symbol for all functions \[
x \mapsto \begin{cases} x^r & \text{ if $x > 0$} \\ 0 & \text{ if $x \leq 0$,} \end{cases}
\]
for $r \in \R$, we obtain the language $\mathcal{L}^\text{pow}_\text{an}$ and the corresponding structure $\R^\text{pow}_\text{an}$. If $X$ is $\mathcal{L}_\text{an}^\text{pow}$-definable in $\R^n$, we say that $X$ is power-subanalytic. In \cite{mil}, Miller considers the following reducts of $\R^\text{pow}_\text{an}$. Let $\mathcal{F}$ be a Weierstrass system: a collection $\mathcal{F}_n$ of real analytic functions $\R^n \to \R$ for each $n \in \N$, containing the polynomials in $n$ variables, for all $n$ which is closed under several operations such as the ring operations, composition and Weierstrass preparation  (for a precise definition, see \cite{mil}). If one adds to the language $\mathcal{L}_r$ symbols for each $\tilde{f}$, where $f \in \mathcal{F}$, one obtains the language $\mathcal{L}_\mathcal{F}$. In particular, the easiest examples are $\mathcal{L}_r$ and $\mathcal{L}_\text{an}$, corresponding to adding no functions and all subanalytic functions respectively. For any Weierstrass system $\mathcal{F}$, one can consider its field of exponents: 
\[
K = \{ r \in \R \mid x \mapsto (1+x)^r \in \mathcal{F} \}.
\]
(Miller shows in his paper that this is indeed a field.) Note that for the smallest Weierstrass system, i.e. just polynomials in $n$ variables, this field is $\Q$ and thus any field of exponents $K$ contains $\Q$. For the largest Weierstrass system, i.e. consisting of all restrictions of all analytic functions (as above), this field is $\R$. Denote $\mathcal{L}^K_\mathcal{F}$ the language obtained by adding to $\mathcal{L}_\mathcal{F}$ a symbol for all power maps $x \mapsto x^r$ for $r \in K$ (as above), where $K$ can be any subfield of $\R$. From now on, definability is with respect to the language $\mathcal{L}_\mathcal{F}^K$, where $K$ is a subfield of the field of exponents of $\mathcal{F}$. The corresponding structure $\R_\mathcal{F}^K$ is o-minimal.

\begin{definition}[Cell]\label{cell}
 A cell in $\R^m$ is a set of the following form: $$\{ (x_1,\ldots,x_m) \in \R^m \mid \alpha_i(x_{< i}) \, \Box_{i1} \, x_i \,\Box_{i2}\, \beta_i(x_{<i}), i = 1,\ldots,m\}$$ where $x_{<i} = (x_1,\ldots,x_{i-1})$, $\alpha_i$ and $\beta_i$ are continuous definable functions, with $\alpha_i < \beta_i$, and $\Box$ is `no condition' or the conditions $<$ or $=$, where $\Box_{i2}$ is $<$ or no condition if $\Box_{i1}$ is equality. Obviously, a cell is open in $\R^m$ if and only if $\Box_{i1}$ is not equality for all $i$. In that case we call $\alpha_i$ and $\beta_i$ the walls of $x_i$.
 \end{definition}
 
 Suppose that $C$ is a cell in $\R^n$. Then, up to reordering the variables if necessary, we may suppose that for $i = 1,\ldots,m$ ($m \leq n$) the condition $\Box_{i1}$ is inequality and for the last $n-m$ variables $\Box_{i1}$ is equality. In this way, we see that any cell in $\R^n$ corresponds to the graph of a definable function $f:U \subset \R^{m} \to \R^{n-m}$. Combining this with the fact that an open cell in $\R^n$ is definably homeomorphic to $(0,1)^n$, we get that a cell in $\R^n$ is the same as the graph of a definable function $(0,1)^{m} \to \R^{n-m}$. The number $m$ is the dimension of the cell and thus is nothing more than counting how many times $\Box_{i1}$ is inequality. We will do many manipulations with these functions and thus end up in general with definable maps $U \subset (0,1)^m \to \R^{n-m}$, where $U$ is an open cell in $(0,1)^m$. To conclude, because any definable set in an o-minimal structure is a finite union of cells, it suffices to prove the main theorem\green{, where we assume $X_{T}$ to be a definable family of $m$-dimensional subsets of $[-1,1]^{n}$, with $T \subset \R^{k}$ the set of parameters,} in the case that $X_T$ is the graph of a definable map $f: T \times (0,1)^m \to [-1,1]^{n-m}$. (If $X$ does not depend on parameters, $T = \R^0$, which is by definition a set with only one point.) Later on, our map will be of the form $f: C \subset T \times (0,1)^m \to [-1,1]^{n-m}$, where the fibers $C_t$ of $C$ are open cells in $(0,1)^m$ for any $t \in T$ and we will denote $f_t$ for the map $C_t \to [-1,1]^{n-m}$ given by $x \mapsto f(t,x)$.
 
 \vspace{0.3cm} In this section \green{$T$ is a definable subset of $\R^{k}$ and} we will often use the following notation. Suppose $U \subset \R^m$. Then we denote by $\overline{U}$ the topological closure of $U$ in $\mathbb{R}^m$ endowed with the standard topology. Furthermore, for any $i \in \{2,\ldots,m\}$, $\pi_{<i}: T \times \R^m \to T \times \R^{i-1}$ denotes the map
 \[
 (t,x_1,\ldots,x_m) \mapsto (t,x_1,\ldots,x_{i-1})
 \]
and $\pi_{<1}(t,x) = t$. The following two definitions coincide with \cite[Definition 4.4.1]{unif}.
 
\begin{definition}[Centre of a cell]
Suppose that $C$ is a cell in $\R^m$. A definable continuous map $\theta: \pi_{<m}(C) \to \R$ is called a centre for $C$ if its graph and $\overline{C}$ are disjoint or if its graph is contained in $\overline{C} \setminus C$\green{, and} $\theta$ is identically zero or $\theta(x_{<m}) \sim x_m$ for all $x \in C$, that is, there exists some $\epsilon \in (0,1)$ such that for all $x \in C$: $$\epsilon x_m \leq \theta(x_{<m}) \leq \epsilon^{-1}x_m.$$
\end{definition}

The next definition motivates why $\theta$ is called the centre of a cell.

\begin{definition}[Prepared with centre]
A bounded definable function $f: C \to \R$, where $C$ is a cell in $\R^m$, is called prepared with centre $\theta$ if $\theta$ is a centre for $C$ and $f$ can be written in the following way: 
\[
f(x) = b_j(x)F(b(x)),
\]
where $b: C \to \R^N$ (for some $N \in \N$) is bounded, $b_j$ is a component function of $b$ and the component functions $b_i$, $i \in \{1,\ldots,N\}$, of $b$ are of the following form: \[
a_i(x_{<m})|x_m-\theta(x_{<m})|^{r_i}\]
with $r_i \in K$ and $a_i: \pi_{<m}(C) \to \R$ definable and, finally, $F$ is a non-vanishing analytic function on an open neighborhood of $\overline{b(C)}$. We call $b$ the associated bounded range map of $f$. A map $f: C \to \R^n$ is prepared with centre $\theta$ if all of its component functions are prepared with centre $\theta$ and moreover have the same bounded range map $b$.
\end{definition}

We now state the preparation theorem as in \cite[Proposition 4.4.2]{unif}, which follows from the main theorem of \cite{mil} by \cite[Lemma 3.7 and Lemma 4.4]{mil}.

\begin{theorem}\label{weierstrass}
Suppose that $f: X \subset \R^m \to \R^n$ is a bounded definable map. Then there exists a finite partition of $X$ into cells $C_i$ with centre $\theta_i$ such that the restriction $f|_{C_i}$ of $f$ to $C_i$ is prepared with centre $\theta_i$.
\end{theorem}

We follow \cite[Section 4]{unif}, but in our proofs, we avoid transforming parameters, i.e. we prove everything uniformly. The main reason we do this is to clearly point out that the power substitution in the proof of \cite[Theorem 2.1.3]{unif} does not require a power substitution in the parameter variables. The idea here is that we stop the inductive proof of \cite[Theorem 4.3.1]{unif} when we reach the parameter variables (see Theorem (\ref{parameter})) and then apply a (uniform) power substitution (Theorem (\ref{wallsub})).  To achieve this, we slightly adjusted some definitions , namely the notion of an a-b-m function in \cite{unif}, which is here replaced by `prepared in $x$' (see below). The main theorem will be deduced from Theorem (\ref{parameter}). In this theorem we parameterize $X_T$ with maps that satisfy the conditions of Proposition (\ref{powersub}). Finally we conclude by Lemma (\ref{ACsub}). 

\begin{definition}\label{prepx}
Suppose that $C$ is a cell in $\R^k\times(0,1)^m$ and denote an element of $C$ as a tuple $(t,x)$. A bounded definable function $f: C \to \R$ is prepared in $x$ if it can be written as: $$f(t,x) = b_j(t,x)F(b(t,x))$$ where $b: C \to \R^N$ (for some $N \in \N$) is bounded, $b_j$ is a component function of $b$, any component function $b_i$, $i \in \{1,\ldots,N\}$, of $b$ is of the form: $$b_i(t,x) = a_i(t)x^{\mu_i} = a_i(t) \prod_{l = 1}^m x_l^{\mu_{i,l}}$$ for some definable function $a_i: \pi_{<k+1}(C) \to \R$, $\mu_i \in K^m$ and $F$ is analytic and non-vanishing on an open neighborhood of $\overline{b(C)}$. We call $b$ the associated bounded monomial map of $f$. A bounded definable map $f: C \to \R^n$ is prepared in $x$ if all of its component functions are prepared in $x$ and have the same associated bounded monomial map $b$.
\end{definition}

Note that $k$ can be zero. In that case, up to a scalar $a$, the associated bounded monomial map $b$ is a monomial and the definition of prepared in $x$ coincides with the definition of an analytic-bounded-monomial map of \cite{unif}. We now show some properties of these maps that relate to the previous section. 

\begin{lemma}\label{bmweaklymild}
Let $b: U \subset (0,1)^m \to \R$, where $U$ is open, be given by $b(x) = ax^\mu$ for some $a \in \R$ and some $\mu \in \R^m$. If $b$ is bounded, then $b$ is weakly $(\green{M},B,0)$-mild, where $\green{M} = \max(1,|\mu_1|,\ldots,|\mu_m|)$ and $B = \sup_U b(x)$.
\end{lemma}

\begin{proof}
Clearly, for any $\nu \in \N^m$, by the form of $b$, we have:
\[
b^{(\nu)}(x) = c(\nu,\mu)\frac{b(x)}{x^\nu},
\]
for some constant $c(\nu,\mu)$ depending on $\nu$ and $\mu$, which remains to be bounded as desired. It is easy to see that computing the next order derivative would multiply $c(\nu,\mu)$ with a factor that is at most $M + |\nu|$ in absolute value. Therefore, we find that:
\[
|c(\nu,\mu)| \leq M(M+1)\cdots(M+|\nu|-1) \leq M^{|\nu|}|\nu|!.
\]
\end{proof}

\begin{lemma}\label{prepxweak}
Suppose that $f: U \subset T \times (0,1)^m \to \R$ is prepared in $x$ with associated bounded monomial map $b$, where for any $t \in T$ the fiber $U_t$ is open. Then there exist  $A,B>0$ such that for any $t \in T$ the map $f_t$ is weakly $(A,B,0)$-mild. Moreover, if the $C^1$-norm of the associated bounded monomial map $b_t$ of $f_t$ is bounded independently of $t$, then there exist $A,B>0$ such that for any $\beta \in \N^m$ with $|\beta| \leq 1$ and $t \in T$ the map $f^{(\beta)}_t$ is weakly $(A,B,0)$-mild.
\end{lemma}
\begin{proof}	
Since $f$ is prepared in $x$, we have that $f(t,x) = b_j(t,x)F(b(t,x))$ as in Definition (\ref{prepx}). By Lemma (\ref{bmweaklymild})\green{, there exist $A,B>0$, such that} for any $t \in T$ the map $b_t$ is weakly $(A,B,0)$-mild. The constant $A$ does not depend on $t$ and, since $b$ is bounded, we may suppose $B$ does not depend on $t$. Finally, by Proposition (\ref{thmweakmild}) and Lemma (\ref{unitmild}), for any $t \in T$ we have that $F \circ b_t$ is weakly $(A,B,0)$-mild for some $A,B > 0$ that do not depend on $t$ and we can conclude by Proposition (\ref{product}) \green{that for any $t \in T$, $f_{t}$ is weakly $(A,B,0)$-mild for some $A,B>0$ that do not depend on $t$.}

Now fix any $t \in T$ and $1 \leq i \leq m$.  By the product rule we have $$\frac{\partial f_t}{\partial x_i} = \frac{ \partial b_{j,t}}{\partial x_i}(F\circ b_t) + \sum_{l = 1}^N b_{j,t}\green{\left(\frac{\partial F}{\partial b_{l,t}}\circ b_{t}\right)}\frac{\partial b_{l,t}}{\partial x_i}.$$ Since we can bound the derivatives of $b_t$ independently of $t$, this is for any $t$ a sum of weakly $(A,B,0)$-mild functions for some $A,B>0$ independent of $t$ and thus is also weakly $(A,B,0)$-mild (for possibly larger $A>0$ and $B>0$).
\end{proof}

This lemma motivates the following definition.

\begin{definition}
Suppose that $f: U \subset \R^k \times \R^m \to \R^n$ is a family of functions. If there exists a $B$ such that for all $t \in \pi_{<1}(U)$ and $x \in U_t$ we have that $|f_t(x)|<B$, we say that $f$ is bounded in $x$. More generally, if $r$ is a natural number, we say $f$ is $C^r$-bounded in $x$ if  there exists a $B$ such that for any $t \in \pi_{<1}(U)$ the map $f_t$ is $C^r$ and for any $x \in U_t$ the $C^r$-norm of $f_t$ is bounded by $B$.
\end{definition}

As an example, let $T = (0,1)$ and consider the family of maps $f_T: (t,x) \mapsto t^2/x$ on the cell given by \begin{align*} 0 < t < 1, \\ t < x < 1. \end{align*} Then for a fixed $t \in T$, the map $f_t$ is in fact mild, but the upper bound depends on $t$. However, one sees that $f$ is $C^1$-bounded in $x$.

Lemma (\ref{prepxweak}) shows that if $f$ is prepared in $x$ such that its associated bounded monomial map $b$ is $C^1$-bounded in $x$, then we can apply Proposition (\ref{powersub}) to $f$. The following result, which we need in the proof of Proposition (\ref{wallsub}), is a more refined version of Proposition (\ref{powersub}) for the bounded monomial map $b$. The technique of the proof is similar, but will use the particular form of $b$.

\begin{proposition}\label{bmpower}
Let $b: U \subset (0,1)^m \to \R$, where $U$ is open, be given by $b(x) = ax^\mu$ for some $a \in \R$ and $\mu \in \R^m$. Suppose that $b$ is $C^1$-bounded. Let $r>0$ be an integer and $\phi : (0,1)^m \to (0,1)^m$ be the map given by:
\[
(x_1,\ldots,x_m) \mapsto (x_1^{n_1},\ldots,x_m^{n_m}),
\]
where for all $i \in \{1,\ldots,m\}$, we have that $n_i \geq r^l$ for some integer $l \geq 1$ and set $V = \phi^{-1}(U)$. Then $(b \circ \phi)^{1/r^{l-1}}\green{: V \to \R}$ is $(\green{\max(n_{1},\ldots,n_{m})}A,B,0)$-mild up to order $r$, for some $A,B>0$ depending on $b$ only.
\end{proposition}
\begin{proof}
For simplicity, we assume that $a = 1$. We have that $(b \circ \phi)(x) = x^{N\mu}$, where $N\mu = (n_i\mu_i)_i$. We apply Theorem (\ref{thmfaa}) to the composition $\green{(b \circ \phi)^{1/r^{l-1}}} =  b^{1/r^{l-1}} \circ \phi$ to get:
\[
|(b^{1/r^{l-1}} \circ \phi)^{(\nu)}(x)| \leq \sum_{1 \leq |\lambda| \leq n} |(b^{1/r^{l-1}})^{(\lambda)}(\phi(x))| \sum_{s = 1}^n \sum_{p_s(\nu,\lambda)} \nu! \prod_{j = 1}^s \frac{|(\phi^{(l_j)}(x))^{k_j}|}{k_j!(l_j!)^{|k_j|}}
\]
for any $x \in V$. We again compute the power of $x_i$ in a fixed term of this sum as in the proof of Proposition (\ref{powersub}). The product of the $|(\phi^{(l_j)}(x))^{k_j}|$ gives a power $\sum_{j = 1}^s (n_i-l_{j,i})k_{j,i}$. \green{Using the notation of Lemma (\ref{bmweaklymild}),} we have that
\[
(b^{1/r^{l-1}})^{(\lambda)}(\phi(x)) = c(\lambda,r^{-(l-1)}\mu) \frac{x^{(1/r^{(l-1)})N\mu}}{x^{N\lambda}},
\]
where $N\lambda = (n_i\lambda_i)_i$ and $c(\lambda,r^{-(l-1)}\mu)$ is a constant depending on $\lambda,r,l$ and $\mu$. We have that $x_i^{\sum_{j = 1}^s ((n_i-l_{j,i})k_{j,i}) - n_i \lambda_i} = x_i^{-\sum_{j = 1}^s l_{j,i}k_{j,i}}$. \green{Let $k \in \{1,\ldots,m\}$ be such that $x_{k} = \min_{i}\{x_{i} \mid \lambda_{i} \neq 0\}$. We may then assume that $\mu_{k} \neq 0$ for the computations below, since if $\mu_{k} = 0$, the partial derivative of $b^{1/r^{l-1}}$ with respect to $\lambda$ is zero since $\lambda_{k} \neq 0$.} Then we bound the product over all $x_i$ in terms of $x_k$:
\[
\prod_{i = 1}^m x_i^{-\sum_{j = 1}^s l_{j,i}k_{j,i}} \leq x_k^{-|\nu|}.
\]
Finally assume $|\nu| \leq r$. Using that $n_k \geq r^l$, we further bound:
\[
\frac{x^{(1/r^{(l-1)})N\mu}}{x_k^{|\nu|}} \leq \frac{x^{(1/r^{(l-1)})N\mu}}{x_k^{r}} = \left( \frac{x^{N\mu}}{x_k^{r^l}} \right)^{1/r^{l-1}} \leq  \left( \frac{x^{N\mu}}{x_k^{n_k}} \right)^{1/r^{l-1}}.
\]
Finally, observe that
\[
\frac{x^{N\mu}}{x_k^{n_k}} = \frac{1}{\mu_k}\left(\frac{\partial}{\partial x_k} b\right)(\phi(x)),
\]
which is bounded since $\phi(x) \in U$ and $b$ is $C^1$-bounded on $U$. This bound only depends on $b$ and the $r^{l-1}$-th root of this bound can be bounded by some constant that only depends on $b$\green{, not on $r$ or $l$}. As in the proof of Lemma (\ref{bmweaklymild}), we can bound $|c(\lambda,r^{-(l-1)}\mu)|$ by $M^{|\lambda|}|\lambda|!$, where $M>0$ \green{can be further bounded to only depend on $b$. One now finishes the proof similarly as the (end of) the proof of Proposition (\ref{powersub}).}
\end{proof}

The next proposition is one of the two main ingredients for the proof of the main theorem.

\begin{proposition}\label{wallsub}
Suppose that $f: C \subset T \times (0,1)^m \to [-1,1]$ is prepared in $x$ such that the associated bounded monomial $b$ is $C^1$-bounded in $x$. Suppose moreover that for any $t \in T$, the cell $C_t$ is open and that the walls of $C$ are also prepared in $x$ with associated bounded monomial maps that are $C^1$-bounded in $x$. Let $r>0$ be any integer and consider the map $\phi_r: T \times (0,1)^m \to T \times (0,1)^m$ given by $$\phi_r(t,x) = (t,x_1^{r^m},x_2^{r^{m-1}},\ldots,x_m^{r}).$$ Then $C_r = \phi_r^{-1}(C)$ is a cell and there exists an $A>0$ such that for any $t \in T$ the walls of the open cell $C_{r,t}$ of $x_1$ are $(A,0)$-mild, for $i = 2,\ldots,m$, the walls of $x_i$ are $(r^mA,0)$-mild up to order $r$ and $(f \circ \phi_r)_t: C_{r,t} \to [-1,1]$ is $(r^mA,0)$-mild up to order $r$.
\end{proposition}

\begin{proof}
Fix some $t \in T$ and consider a wall of the cell $C_{t}$, say $\alpha_i$ of $x_i$. If $i = 1$, the walls are constant and there is nothing to show, so suppose $i \geq 2$. Then $\alpha_i(x_{<i}) = b_{j,t}(x_{<i})F(b_t(x_{<i}))$ \green{for some $b$ and $F$ as in Definition (\ref{prepx})}, since $\alpha_i$ is prepared in $x$. It follows that the corresponding wall of $x_i$ in the cell $C_{r,t}$ is given by:
\[
\alpha_{r,t}(x_1,\ldots,x_{i-1}) = \sqrt[r^{m-i+1}]{\green{\alpha_{i,t}}(x_1^{r^m},\ldots,x_{i-1}^{r^{m-i+2}})}.
\]
By \green{a uniform version of} Proposition (\ref{bmpower}), we have that $\sqrt[r^{m-i+1}]{b_{j,t}(x_1^{r^m},\ldots,x_{i-1}^{r^{m-i+2}})}$ is $(r^mA_1,B_1,0)$-mild up to order $r$ for some $A_1,B_1 > 0$ independent of $t$ and $r$. \green{ $A_{1}$ does not depend on $t$ since it only depends on the exponent of $x$ in $b_{j}$, which does not depend on $t$ by Definition (\ref{prepx}). To show that $B_{1}$ does not depend on $t$, one should dismiss the assumption that $a = 1$ in the proof of (\ref{bmpower}), but can later use that $b_{j}$ is $C^{1}$-bounded in $x$ to find $B_{1}$ independent of $t$ (and also $r$ and $l$).}

Since $F$ is a non-vanishing analytic function on some open neighborhood of $\im(b_t)$, we may suppose that there is some $S>1$ such that $\im(F) \subset (1/S,S)$. On this domain, the function $\sqrt[r^{m-i+1}]{}$ is $(S,S,0)$-mild. Hence, $\sqrt[r^{m-i+1}]{} \circ F$ is an $(A_F,B_F,0)$-mild map for some $A_F,B_F>0$ by Lemma (\ref{unitmild}) and Corollary (\ref{compformula}), where $A_F$ and $B_F$ only depend on $F$, not on $t$ or $r$. Next, $b_t(x_1^{r^m},\ldots,x_{i-1}^{r^{m-i+2}})$ is $(r^mA',B',0)$-mild up to order $r$ by Proposition (\ref{powersub}) for some $A',B' > 0$ depending only on $b$,not on $t$ or $r$. Applying \green{the up to order $r$ version of} Corollary (\ref{compformula}), we see that $\sqrt[r^{m-i+1}]{F(b_t(x_1^{r^m},\ldots,x_{i-1}^{r^{m-i+2}}))}$ is $(r^mA_2,B_2,0)$-mild up to order $r$ for some $A_2,B_2 > 0$ depending only on $\alpha_i$, not on $t$ or $r$.

By Proposition (\ref{product}), we see that $\alpha_{r,t}$ is $(r^mA,B,0)$-mild up to order $r$ for some $A,B > 0$ independent of $t$ and $r$. Since $\im(\alpha_{r,t}) \subset (0,1)$, after possibly enlarging $A$, we can conclude that $\alpha_{r,t}$ is $(r^mA,0)$-mild up to order $r$ for some $A>0$ independent of $t$ and $r$. Finally, by Proposition (\ref{powersub}) and Lemma (\ref{prepxweak}), the function $(f \circ \phi_r)_t$ is $(r^mA'',B'',0)$-mild up to order $r$ for some $A'',B''>0$ independent of $t$ and $r$. Since $|f| \leq 1$, after possibly enlarging $A''$, we have that $(f \circ \phi_r)_t$ is $(r^mA'',0)$-mild up to order $r$. We conclude after redefining $A$ to be the maximum over all $A$ corresponding to the mildness of the walls and $A''$ of the mildness of $(f \circ \phi_r)_t$.
\end{proof}

\green{We will now show one can obtain the conditions of this proposition.} It is a parameterization result, a slight modification of \cite[Theorem 4.3.1]{unif}, that will allow us to easily deduce the main theorem.

\begin{theorem}\label{parameter}
Suppose that $X_T$ is the graph of a definable function $\green{\varphi}: C \subset T \times (0,1)^m \to [-1,1]^{n-m}$, where $C$ is \green{open in $T \times (0,1)^{m}$}. Then there exist finitely many definable maps $$f_l: C_l \to X_T$$ such that:
\begin{enumerate}
\item $\bigcup_{l} \text{Im}(f_l) = X_T$,
\item for each $l$, $C_l$ is an open cell in $T_l \times (0,1)^m$, where $T_l$ is a cell contained in $T$,
\item for each $l$ and for any $(t,x) \in C_l$, $f_l(t,x) \in X_t$, thus $f_l$ is a family of maps $C_{l,t} \to X_t$ with $C_{l,t}$ open in $(0,1)^m$,
\item each $f_l$ is prepared in $x$ and the associated bounded monomial map of $f_l$ is $C^1$-bounded in $x$,
\item the walls of all $C_l$ are prepared in $x$ and their associated bounded monomial maps are $C^1$-bounded in $x$.
\end{enumerate}
\end{theorem}

\begin{proof}
The proof uses induction on $m$. The case $m=0$ is trivial. \green{By o-minimality, we may assume that there exist finitely many $f_{l}: C_{l} \subset T \times (0,1)^{m} \to X_{T}$, satisfying properties $1$, $2$ and $3$. For the rest of the proof, we focus on one such map $f_{l}$. Next, we apply Theorem (\ref{weierstrass}) to $f_{l}$ such that we may suppose that $f_{l}$ is prepared on $C_{l}$ with centre $\theta_{l}$. Note that some of the cells obtained by applying (\ref{weierstrass}) might not be open, but then it has lower dimension, therefore can be considered separately and can be ignored by induction. For this reason, whenever we further partition $C_{l}$, we always assume the cells are open.}
%

\green{Next}, we will show that we may assume $\theta_l = 0$ on $C_l$. Since $x_m \sim \theta_l(x_{<m})$ on $C_l$, there exists some $\epsilon \in (0,1)$ such that $\epsilon x_m < \theta_l(t,x_{<m}) < (1/\epsilon) x_m$ for all $(t,x) \in C_l$. Up to finite partitioning, if necessary, we may assume that either $\theta_l(t,x_{<m}) > x_m$ for all $(t,x) \in C_l$ or $\theta_l(t,x_{<m}) < x_m$ for all $(t,x) \in C_l$. In the first case, the other is similar, \green{consider the \emph{cell}}
\[
\tilde{C}_l = \{ (t,x) \in T \times (0,1)^m \mid (t,x_{<m},\frac{-1}{\epsilon}x_m + \theta_l(t,x_{<m})) \in C_l \}
\]
and define \green{$\tilde{f}_l: \tilde{C}_l \to X_{T}$ by $ \tilde{f}_{l}(t,x) =  f_l(t,x_{<m},\frac{-1}{\epsilon}x_m+\theta_l(t,x_{<m}))$}. Then $\tilde{f}_l$ is prepared on $\tilde{C}_l$ with centre zero. Moreover, properties $1$, $2$ and $3$ are still satisfied. Thus, from now on we will assume additionally that $f_l$ is prepared on $C_l$ with centre zero.

Next we show that, up to finite partitioning, if necessary, the associated bounded range map $b_l = (b_{l,1},\ldots,b_{l,N})$ of $f_l$ is $C^1$ and $| \partial b_l / \partial x_m | \leq 1$. That we may suppose that $b_l$ is $C^1$ is a classical consequence of the cell decomposition theorem. Up to further partitioning using o-minimality, we may suppose that there exists some $i \in \{1,\ldots,N\}$ such that for any other $i' \in \{1,\ldots,N\}$ we have: $|\partial b_{l,i} / \partial x_m | \geq |\partial b_{l,i'} / \partial x_m|$ on $C_l$ and that either $|\partial b_{l,i} / \partial x_m | \geq 1$ or $|\partial b_{l,i} / \partial x_m | < 1$ on $C_l$. The second case is exactly what we want. In the first case, we do a change of variables. To ensure that we recover a map $T \times (0,1)^m \to [-1,1]^{m-n}$, we first have to further partition $C_l$ such that $b_{l,i}$ is either identically $-1$,$0$ or $1$, $b_{l,i} > 0$ or $b_{l,i} < 0$. We only have to consider the last two cases. Suppose we have $
b_{l,i} > 0$ (the case $b_{l,i} < 0$ is similar up to changing a sign). Once more using o-minimality, we may assume that for fixed $(t,x_1,\ldots,x_{m-1})$ the map $x_m \mapsto b_{l,i}(t,x_1,\ldots,x_{m-1},x_m)$ is injective and it follows that the map 
\[
\phi: C_l \to \phi(C_l) : (t,x) \mapsto (t,x_1,\ldots,x_{m-1},b_{l,i}(t,x))
\]
is invertible. Set $\bar{C}_l = \phi(C_l)$ and $\bar{f}_l = f_l \circ \phi^{-1}$. Then $\bar{C}_l$ is an open cell in $T \times (0,1)^m$ and $\bar{f}_l$ satisfies the same properties as $f_l$, in particular it is prepared in $x_m$ with bounded range map $\bar{b_l}$ but moreover $|\partial \bar{b}_{l,i} / \partial x_m | \leq 1$ for all component functions $\bar{b}_{l,i}$ of $\bar{b}_l$ (by the chain rule and the choice of $i$). \vspace{0.3cm}

So up to now, we may suppose we have finitely many maps $f_l:C_l \to X_T$ satisfying the first three properties, that are prepared in $x_m$ and the associated bounded range map $b_l$ of $f_l$ is $C^1$ and $|\partial b_{l,i} / \partial x_m| \leq 1$. Thus, for any $i \in \{1,\ldots,N\}$, the component function $b_{l,i}$ of $b_l$ is given by:
\[
b_{l,i}(t,x) = a_{l,i}(t,x_{<m})x_m^{r_{l,i}}.
\]

Denote $\alpha_m$ for the wall of $C_l$ bounding the variable $x_m$ from below and $\beta_m$ for the wall of $C_l$ bounding $x_m$ from above. Up to further partitioning, we may suppose $\alpha_m \neq 0$ on $C_l$ and we will also suppose that $r_{l,i} \neq 0$ for every $i \in \{1,\ldots,N\}$. We will come back to this at the end of the proof. Now consider the following maps:
\begin{align*}
&h_\alpha: \pi_{<m}(C_l) \to [-1,1]^N : (t,x_{<m}) \mapsto \lim_{x_m \to \alpha_m(t,x_{<m})} b_l(t,x_{<m},x_m) \\
&g_\alpha: \pi_{<m}(C_l) \to [-1,1]^N : (t,x_{<m}) \mapsto \lim_{x_m \to \alpha_m(t,x_{<m})} \frac{\partial b_l}{\partial x_m}(t,x_{<m},x_m)
\end{align*}
and define $h_\beta$ and $g_\beta$ analogously. \green{Note that $g_{\alpha}$ and $g_{\beta}$ are well defined by our effort to bound $|\partial b_{l,i}/\partial x_{m}|$ by 1.} For any $i \in \{1,\ldots,N\}$, by the form of $b_{l,i}$, the component functions of $h_{\alpha,i}$ and $g_{\alpha,i}$ (similarly for $h_\beta$ and $g_\beta$) become:
\begin{align*}
&h_{\alpha,i}(t,x_{<m}) = a_{l,i}(t,x_{<m}) \alpha_m(t,x_{<m})^{r_{l,i}},  \\
&g_{\alpha,i}(t,x_{<m}) = r_{l,i}a_{l,i}(t,x_{<m}) \alpha_m(t,x_{<m})^{r_{l,i}-1} \quad (r_{l,i} \neq 0).
\end{align*}
Now define the map $F: \pi_{<m}(C_l) \to \text{Im}(F)$ whose component functions are $\alpha_m,\beta_m,h_\alpha,h_\beta,g_\alpha$ and $g_\beta$. Next, apply the induction hypothesis to the graph of $F$. Hence we obtain finitely many maps $\psi_{l,j}:D_{l,j} \to \text{graph}(F)$ satisfying all properties of the theorem. In particular, they are prepared in $x_{<m}$ with associated bounded monomial map $c_{l,j}$ that is $C^1$-bounded in $x_{<m}$. Therefore, by their form, the maps $\psi_{l,j}$ are $C^1$-bounded in \green{$x_{<m}$}. Note that $\psi_{l,j}(t,x_{<m})_{<m} \in \pi_{<m}(C_l)$ and thus, for instance, $\alpha_m(\psi_{l,j}(t,x_{<m})_{<m})$ is a component function of $\psi_{l,j}$. Now set
\begin{equation*}
C_{l,j} = \{ (t,x_{<m},x_m) \in D_{l,j} \times (0,1) \mid (\psi_{l,j}(t,x_{<m})_{<m},x_m) \in C_l \}.
\end{equation*}
By construction, these cells have all the properties we want. Finally define $f_{l,j}:C_{l,j} \to X_T$ by:
\begin{equation*}
f_{l,j}(t,x_{<m},x_m) = f_l(\psi_{l,j}(t,x_{<m})_{<m},x_m).
\end{equation*}
By the form of $f_l$, it remains to show that for any $i \in \{1,\ldots,N\}$, $b_{l,i}(\psi_{l,j}(t,x_{<m})_{<m},x_m)$ is prepared in $x$ and is $C^1$-bounded in $x$. Thus, let $b_{l,i}$ be a component function of $b_l$ and suppose $r_{l,i} < 0$ (in the case $r_{l,i} > 0$, use $\beta$ instead of $\alpha$ in the calculations). We have that:
\begin{align*}
b_{l,i}(\psi_{l,j}(t,x_{<m})_{<m},x_m) &= a_{l,i}(\psi_{l,j}(t,x_{<m})_{<m})x_m^{r_{l,i}} \\
&= h_{\alpha,i}(\psi_{l,j}(t,x_{<m})_{<m})\left(\frac{x_m}{\alpha_m(\psi_{l,j}(t,x_{<m})_{<m})}\right)^{r_{l,i}}.
\end{align*}
By the construction, $h_{\alpha,i}(\psi_{l,j}(t,x_{<m})_{<m})$ and $\alpha_m(\psi_{l,j}(t,x_{<m})_{<m})$ are prepared in $x_{<m}$ with associated bounded monomial map $c_{l,j}$ that is $C^1$-bounded in $x_{<m}$. It follows that $b_{l,i}(\psi_{l,j}(t,x_{<m})_{<m},x_m)$ is prepared in $x$. We now look at the $C^1$-norm. To simplify notation, denote $y = \psi_{l,j}(t,x_{<m})_{<m}$. Let $s \in \{1,\ldots,m-1\}$, then:
\begin{align*}
\frac{\partial}{\partial x_s} \left(b_{l,i}(y,x_m)\right) &= \frac{\partial}{\partial x_s}\left(h_{\alpha,i}(y)\left(\frac{x_m}{\alpha_m(y)}\right)^{r_{l,i}}\right) \\
&= \frac{\partial}{\partial x_s}(h_{\alpha,i}(y))\left(\frac{x_m}{\alpha_m(y)}\right)^{r_{l,i}} + h_{\alpha,i}(y)\frac{\partial}{\partial x_s}\left(\left(\frac{x_m}{\alpha_m(y)}\right)^{r_{l,i}}\right).
\end{align*}
The first term is bounded in $x$ since \green{$h_{\alpha,i}(y) = h_{\alpha,i}(\psi_{i,l}(t,x_{<m})_{<m})$} is $C^1$-bounded in $x_{<m}$ \green{(because it is prepared in $x_{<m}$, with $C^{1}$-bounded associated bounded monomial map)} and $(x_m/\alpha_m(y))^{r_{l,i}} <1$ since $x_m > \alpha_m(y)$ and $r_{l,i} < 0$. We further compute the last term:
\begin{align*}
h_{\alpha,i}(y)\frac{\partial}{\partial x_s}\left(\left(\frac{x_m}{\alpha_m(y)}\right)^{r_{l,i}}\right) &= -a_{l,i}(y)\alpha_m(y)^{r_{l,i}}\frac{r_{l,i}}{\alpha_m(y)} \left(\frac{x_m}{\alpha_m(y)}\right)^{r_{l,i}}\frac{\partial}{\partial x_s}\left(\alpha_m(y)\right) \\
&= -g_{\alpha,i}(y) \left(\frac{x_m}{\alpha_m(y)}\right)^{r_{l,i}} \frac{\partial}{\partial x_s}(\alpha_m(y)).
\end{align*}
Hence we see that also this term is bounded in $x$. Since $(\partial b_l / \partial x_m)$ was already bounded, we obtain that $b_{l,i}(\psi_{l,j}(t,x_{<m})_{<m},x_m)$ is $C^1$-bounded in $x$.

To conclude we explain the cases $\alpha_m = 0$ and $r_{l,i} = 0$. If $\alpha_m = 0$, this forces the exponent $r_{l,i}$ to be positive or zero since $b_l$ is bounded. If $r_{l,i} > 0$, we can just use the maps $h_\beta$ and $g_\beta$ as above since $\beta_m > 0$. If $r_{l,i} = 0$, one can use the induction hypothesis on $b_{l,i} = a_{l,i}$ immediately.
\end{proof}

We can now easily prove the main theorem using this parameterization result and Lemma (\ref{ACsub}). It refines the $C^r$-parameterization theorem of \cite{unif} by making the constant $d$ more explicit; we have $d = m^3$.

\begin{theorem}\label{maintheorem}
Suppose that $X_T$ is a definable family of $m$-dimensional subsets in $[-1,1]^n$. Then there exists a constant $c>0$ such that for any integer $r > 0$ and $t \in T$ there is a collection of finitely many analytic maps $$\{f_{r,i,t}:(0,1)^m \to X_t \mid i \in \{1,\ldots,cr^{m^3}\} \}$$ whose $C^r$-norm is bounded by $1$ and such that for any $t \in T$ the ranges of $f_{r,i,t}$, for $i = 1,\ldots,cr^{m^3}$, cover $X_t$. Moreover for each $i$ and $r$, $\{f_{r,i,t} \mid t \in T\}$ is a definable family of maps.
\end{theorem}
\begin{proof}
As we explained below Definition (\ref{cell}), we may suppose that $X_T$ is the graph of a definable function $T \times (0,1)^m \to [-1,1]^{n-m}$. By the previous result, we obtain finitely many $f: C \to X_T$, where $C$ is an open cell in $T \times (0,1)^m$. Furthermore, $f$ is prepared in $x$ with associated bounded monomial map $b$ that is $C^1$-bounded in $x$ and also the walls of $C$ are of this form. By Proposition (\ref{wallsub}) using the map 
\[
(t,x_1,\ldots,x_m) \mapsto (t,x_1^{r^m},x_2^{r^{m-1}},\ldots,x_m^r),
\]
we get finitely many maps $f_{r}: C_{r} \to X_T$, where $C_{r}$ is open in $T \times (0,1)^m$. Moreover, there is an $A > 0$ such that for any $t \in T$, the walls $\alpha_1$ and $\beta_1$ of $C_{r,t}$ of $x_1$ are $(A,0)$-mild, for $i = 2,\ldots,m$, the walls $\alpha_i$ and $\beta_i$ of $x_i$ are $(Ar^m,0)$-mild up to order $r$ and $f_{r,t}: C_{r,t} \to X_t$ is $(Ar^{m},0)$-mild up to order $r$.
 
We now inductively map $T \times (0,1)^m$ on the cell $C_r$ as follows. For $i = 1,\ldots,m$, define the map $$\Phi_i: \pi_{<i}(C_r) \times (0,1)^{m-i+1} \to \pi_{< i+1}(C_r) \times (0,1)^{m-i}$$ given by:
 \[
(t,x) \mapsto (t,x_1,\ldots,x_{i-1},((\beta_i - \alpha_i)x_i + \alpha_i)(t,x_{<i}),x_{i+1},\ldots,x_m)
\]
and let $\Phi = \Phi_1 \circ \ldots \circ \Phi_m: T \times (0,1)^m \to C_r$. If $i \geq 2$, since for any $t \in T$, $\alpha_i$ and $\beta_i$ are $(Ar^m,0)$-mild up to order $r$, there is some $A'>0$ such that for any $t \in T$, the map $\Phi_{i,t}$ is $(A'r^m,0)$-mild up to order $r$, where $A'>0$ is possibly larger than $A$ due to the product (\ref{product}) and addition of mild functions. If $i = 1$, $\Phi_{1,t}$ is $(A',0)$-mild, for possibly larger $A'$. Thus we obtain that there is some $A'' > 0$, possibly larger than $A'$ by Corollary (\ref{compformula}), such that for any $t \in T$, the map $(f_r \circ \Phi)_t$ is $(A''r^{m^2},0)$-mild up to order $r$.
 
 Finally, by Lemma (\ref{ACsub}), covering the unit cube with cubes of size $1/(A''r^{m^2})$, we obtain $cr^{m^3}$ maps that parametrize $X_T$ with bounded $C^r$-norm.
\end{proof}

\emph{Remark.} The main obstacle to obtain a better result (e.g. $cr^m$ maps as in \cite{ccs}), is the power substitution using $r^{m}$ to ensure that the walls of the cell are mild up to order $r$ too. Thus it would be interesting to prove a stronger version of Theorem (\ref{parameter}) that has better walls. 

A small improvement that we show now slightly improves the walls. More precisely, if $m \geq 2$, we will show that we can first improve the walls bounding the variable $x_2$ and use a slightly better power substitution with powers up to $r^{m-1}$ (instead of $r^m$).

Suppose we are given a map as a result of Theorem (\ref{parameter}), call it $f: C \to X_T$, where $C$ is an open cell in $T' \times (0,1)^m$ and $T'$ a cell contained in $T$. Now, \green{since} the walls bounding $x_1$ and $x_2$ are prepared in $x$, \green{the first two inequalities defining $C_{t}$ are:}
\begin{align*}
\alpha_1(t) < &\ x_1 < \beta_1(t) \\
a_i(t)x_1^rF(a(t,x_1)) < &\ x_2 < b_j(t)x_1^sG(b(t,x_1)).
\end{align*}

We will now manipulate $f$ and $C$ to ensure that the maps $a_i(t)x_1^r$ and $b_j(t)x_1^s$ become mild for any $t \in T$, by transforming $r$ and $s$ to natural numbers. If they are already mild, there is no power substitution required for these walls and we can use a better version of Proposition (\ref{wallsub}). Suppose that $r-s \geq 0$, the other case is similar. Since $G$ is analytic and non-vanishing on $\text{Im}(b)$, there exists $S>0$ such that $G(b(t,x_1)) \in (1/S,S)$ for all $t$ and $x$. Equivalently: $G(b(t,x_1))/S \in (1/S^2,1)$. Now consider the map $\phi$ given by:
\begin{equation*}
\phi:T \times (0,1)^m \to \text{Im}(\phi): \phi(t,x_1,\ldots,x_m) = (t,x_1,b_j(t)x_1^sSx_2,x_3,\ldots,x_m).
\end{equation*}
Since $b_j(t)x_1^s$ is $C^1$-bounded in $x_1$ by assumption, the map $\phi$ is $C^1$-bounded in $x$. We see that $\phi^{-1}(C)$ is a cell where the walls bounding $x_1$ and $x_2$ are now given by:
\begin{align*}
\alpha_1(t) < &\ x_1 < \beta_1(t) \\
(a_i/b_j)(t)x_1^{r-s}F(a(t,x_1))/S < &\ x_2 < G(b(t,x_1))/S.
\end{align*}
By the form of $\phi$, the walls bounding $x_3,\ldots,x_m$ are still prepared in $x$ and their associated bounded monomial map is $C^1$-bounded in $x$. If $r-s = 0$, the walls are already as desired. Denote by $R$ the smallest integer greater than or equal to $r-s$. We then use the power substitution $$(t,x_1,x_2,\ldots,x_m) \mapsto (t,x_1^{R/(r-s)},x_2,\ldots,x_m)$$ to obtain a cell $\tilde{C}$ and a map $\tilde{f}: \tilde{C} \to X_T$ with the same properties as $f$, i.e. $\tilde{f}$ and all of the walls of $\tilde{C}$ are prepared in $x$ with an associated bounded monomial map that is $C^1$-bounded in $x$. The walls of $\tilde{C}$ bounding $x_1$ and $x_2$ are of the form:
\begin{align*}
\tilde{\alpha}_1(t) < &\ x_1 < \tilde{\beta}_1(t) \\
\tilde{a}_i(t)x_1^NF(\tilde{a}(t,x_1)) < &\ x_2 < G(\tilde{b}(t,x_1))
\end{align*}
where $N \in \N$. We now have that for any $t \in T$, $\tilde{a}_i(t)x_1^N$ is mild. One now proceeds as in the proof of (\ref{maintheorem}) using a slightly modified version of Proposition (\ref{wallsub}), namely with the power substitution
\[
(t,x_1,\ldots,x_m) \mapsto (t,x_1^{r^{m-1}},x_2^{r^{m-1}},x_3^{r^{m-2}},\dots,x_m^r).
\]
In the proof of Theorem (\ref{maintheorem}), we would now have that there is some $A>0$ such that for any $t \in T$, the walls $\alpha_i$ and $\beta_i$ of $x_i$ are $(Ar^{m-1},0)$-mild up to order $r$ and $f_{r,t}$ is $(Ar^{m-1},0)$-mild up to order $r$. It follows that for some possibly larger $A$ and for any $t \in T$, $(f_r \circ \Phi)_t$ is $(Ar^{m(m-1)},0)$-mild up to order $r$ and thus obtain $cr^{m^2(m-1)}$ maps as result.

It would be interesting to know if we could also transform the walls of the other variables in this way. However, currently it is not clear how to achieve this.

\end{document}